\documentclass[12pt]{amsart}
\usepackage{amssymb, amscd, amsmath, amsthm, epsfig, latexsym, enumerate}

\renewcommand{\geq}{\geqslant}

\newtheorem*{thm}{Theorem}
\newtheorem*{lem}{Lemma}
\newtheorem*{cor*}{Corollary}

\begin{document}
\title{Centralizers of torsion in 3-manifold groups}

\author{Jonathan A. Hillman }
\address{School of Mathematics and Statistics\\
     University of Sydney, NSW 2006\\
      Australia }

\email{jonathanhillman47@gmail.com}

\begin{abstract}
We show that if $P$ is a $PD_3$-complex
and $g\in\pi_1(P)$ has finite order $>1$ and infinite centralizer
then $\pi_1(P)$ retracts onto $Z/2Z\oplus\mathbb{Z}$.
If $P$ is an irreducible closed 3-manifold it then follows from 
the Projective Plane Theorem that $P\cong{RP^2}\times{S^1}$.
\end{abstract}

\keywords{centralizer, non-orientable, $PD_3$-complex,  projective plane}

\maketitle
The focus of most research on manifolds has been on the orientable case.
Here we shall point out a curiosity from the realm of non-orientable 
3-manifolds that seems not to have been noticed before.
We shall show that if $M$ is an irreducible closed 3-manifold 
such that $\pi_1(M)$ has an element $g\not=1$ of finite order and with infinite centralizer then $M$ is homeomorphic to ${RP}^2\times{S^1}$.

Our argument is mostly on the broader level of $PD_3$-complexes.
Crisp has shown that if $P$ is a $PD_3$-complex and $\pi=\pi_1(P)$ has 
an element $g\not=1$ of finite order and with infinite centralizer $C_\pi(g)$ 
then $g^2=1$,  
$g$ is orientation-reversing and $C_\pi(g)$ has two ends
\cite[Theorem 17]{Cr}.
(This can be sharpened to give
$C_\pi(g)\cong\langle{g}\rangle\times\mathbb{Z}$ \cite[Corollary 7.10.2]{Hi}.)
Crisp's argument proved more than it asserted,
and we shall extend it to show that $\pi$ retracts onto 
$Z/2Z\times\mathbb{Z}$.
In the 3-manifold case we have further tools,
namely the Projective Plane Theorem and its consequences,
and the assertion in the first paragraph is then a corollary of our main result.

At the end we give alternative arguments due to J. S. Crisp 
(for the $PD_3$ result) and to G. A. Swarup (for the application to 3-manifolds).
I thank them for allowing me to include outlines of their arguments.

\section{the main result}

We begin with an algebraic lemma.

\begin{lem}
Let $K$ be a finitely generated group and let $G$ be the HNN extension 
with presentation $\langle{K,t}\mid{tgt^{-1}=j}\rangle$,
where $g^2=j^2=1$,
$C_K(g)=\langle{g}\rangle$ and $C_K(j)=\langle{j}\rangle$.
Suppose that $g$ and $j$ are not conjugate in $K$.
Then $C_G(g)=\langle{g}\rangle$ (and $C_G(j)=\langle{j}\rangle$).
\end{lem}

\begin{proof}
Let $S_+$ and $S_-$ be left transversals for $\langle{j}\rangle$ 
and $\langle{g}\rangle$ in $K$ (respectively), 
chosen so that $1\in{S_-\cap{S_+}}$.
Each element in the given HNN extension has an unique normal form
$s_0t^{\varepsilon_0}s_1\dots{t^{\varepsilon_{n-1}}}k$,
where $\varepsilon_i=\pm1$, 
$s_i\in{S_+}$ if $\varepsilon_i=1$,
$s_i\in{S_-}$ if $\varepsilon_i=-1$, 
$s_i\not=1$ if 
$\varepsilon_{i-1}=-\varepsilon_i$, 
and $k\in{K}$.
If there is such an element of length $n>0$ which commutes with $g$ then
\[
gs_0t^{\varepsilon_0}s_1\dots{t^{\varepsilon_{n-1}}}k=
s_0t^{\varepsilon_0}s_1\dots{t^{\varepsilon_{n-1}}}kg.
\]
The right hand element is in normal form.
Since $g\not=1$, the element on the left cannot be in normal form
(by the uniqueness of normal forms).
If $\varepsilon_0=1$ then we must have $gs_0=s'j$ for some $s'\in{S_+}$.
But then the normal form for the left hand element would begin $s't\dots$.
Hence $s'=s_0$ and $g=s_0js_0^{-1}$, 
which contradicts one of the hypotheses.
If $\varepsilon_0=-1$ then we must have $gs_0=s''g$ for some $s''\in{S_-}$.
We again see that we must have $s''=s_0$ and so $gs_0=s_0g$.
Hence $s_0=1$, since $C_K(g)=\langle{g}\rangle$,
and we may cancel $t^{\varepsilon_0}=t^{-1}$ from each side,.
If $n=1$ we get $jk=kg$, and so $j=kgk^{-1}$, contrary to hypothesis. 
If $n>1$ we get
\[
js_1t^{\varepsilon_1}\dots{t^{\varepsilon_{n-1}}}k=
s_1t^{\varepsilon_1}\dots{t^{\varepsilon_{n-1}}}kg.
\]
The right hand side is again in normal form.
If $\varepsilon_1=1$ then we must have $js_1=s'j$ for some $s'\in{S_+}$.
But then $s'=s_1$, so $s_1$ commutes with $j$.
Since $S_+\cap{C_K(j)}=1$, we get $s_1=1$.
This is impossible,  since $\varepsilon_0=-\varepsilon_1$.
Therefore $\varepsilon_1=-1$, and so $js_1=s''g$ for some $s''\in{S_-}$.
But then $s''=s_1$, by the uniqueness of normal forms, 
and so $g$ and $j$ are conjugate,  
which is again a contradiction.

Thus $C_G(\langle{g}\rangle)=C_K(\langle{g}\rangle)=\langle{g}\rangle$
(and likewise $C_G(\langle{j}\rangle)=\langle{j}\rangle$).
\end{proof}

Let $P$ be a non-orientable $PD_3$-complex and let $\pi=\pi_1(P)$. 
Let $P^+$ be the orientable double covering space with
$\pi_1(P^+)=\pi^+=\mathrm{Ker}(w)$,
where $w:\pi\to\mathbb{Z}^\times$ is the orientation character.
If $G$ is a subgroup of $\pi$ let $G^+=G\cap\pi^+$.

\begin{thm}
Let $\pi=\pi_1(P)$, where $P$ is a $PD_3$-complex,
and suppose that $g\in\pi$ has finite order and infinite centralizer $C=C_\pi(g)$.
Then $\pi\cong\rho*_{Z/2Z}(Z/2Z\oplus\mathbb{Z})$,
for some subgroup $\rho$, and $\pi$ retracts onto $Z/2Z\oplus\mathbb{Z}$.
\end{thm}

\begin{proof}
We may assume that $P$ is indecomposable.
The element $g$ has order 2 and $w(g)=-1$, 
so $P$ is non-orientable \cite{Cr}.
If $P\simeq{RP^2}\times{S^1}$ there is nothing to prove, 
for then $\pi\cong{Z/2Z\oplus\mathbb{Z}}$.
Otherwise, $\pi=\pi\mathcal{G}$ where $(\mathcal{G},\Gamma)$
is a finite graph of groups in which each vertex group is $FP_2$ 
and has one end,
the edge groups are all of order 2, 
$\pi^+=\mathrm{Ker}(w_1(P))$ is torsion-free
and $\pi\cong\pi^+\rtimes{Z/2Z}$ \cite[Theorem 7.10]{Hi}.

The argument of Crisp's Centralizer Theorem  \cite[Theorem 17]{Cr}
shows more than it asserts:
in fact $g$ stabilizes an edge $e$ of a terminal $\pi$-tree $T$, 
and an element $h\in{C}$ acts as a translation along a line through $e$.
After contracting edges not in the $\pi$-orbit of $e$, 
we obtain a $\pi$-tree $U$ which is a union of translates of $e$, 
and so either $\pi\cong{A*_{\langle{g}\rangle}B}$ with $|A|,|B|>2$
or $\pi\cong{A*_{\langle{g}\rangle}\varphi}$ where $j=\varphi(g)\in{A}$.

If $\pi\cong{A*_{\langle{g}\rangle}B}$ then $\langle{g}\rangle$ 
must be properly contained in its centralizer in one of $A$ or $B$,
since $C$ is generated by $C_A(g)\cup{C_B(g)}$,
by the Normal Form Theorem.
Since $C\cong\langle{g}\rangle\oplus\mathbb{Z}$
\cite[Corollary 7.10.2]{Hi},
it  is properly contained in just one of these.
But since $h$ is a translation it is not in any vertex group
\cite[I.4.11]{DD}.
Therefore we must have 
$\pi\cong{A*_{\langle{g}\rangle}\varphi}$.

If $C_A(j)\not=\langle{j}\rangle$ then $j$ must stabilize some edge $f$ 
of $T$ and there is an element $h\in{C_A(j)}$ which acts 
as a translation along $f$.
The subgroup $A$ is also the quotient 
$\pi/\langle\langle{B^+}\rangle\rangle$,
and thus is a retract of $\pi$.
We may apply the earlier argument to  $A$ and $j$, instead of $\pi$ and $g$.
The rest of our present argument involves a finite induction, 
and does not use the hypothesis that the ambient group $\pi$ 
is the fundamental group of a $PD_3$-complex.
All the intermediate groups arising en route have the form
$\pi\mathcal{H}$, where $(\mathcal{H},\Theta)$ is a graph of groups 
with $\Theta$ a non-empty connected subgraph of $\Gamma$,
vertex groups $H_v=G_v$ for all $v\in{V(\Theta)}$
and all edge groups of order 2.

Since $\pi$ is $FP_2$ it is accessible, 
as are subgroups such as $A$,
and so after a finite number of iterations
we reach a group $G$ which is a retract of $\pi$,
and which is an HNN extension $K*_E\psi$,
where $E$ is an edge group,  $C_K(E)=E$, 
$C_K(\psi(E))=\psi(E)$ and $C_G(E)\cong{Z/2Z}\oplus\mathbb{Z}$.
Let $t$ be the stable letter for this HNN extension.
We may now apply the lemma, with $\langle{g}\rangle=E$ and $j=\psi(g)$.
We conclude that $j=k^{-1}gk$ for some $k\in{K}$.
Let $u=kt$.
Then $ug=gu$ and $G$ has the HNN presentation
$\langle{K,u}\mid{ugu^{-1}=g}\rangle$.
Hence $G\cong{K}*_EC_G(E)$.
It follows easily that $\pi\cong\rho*_EC_G(E)$, for some subgroup $\rho$.
This clearly retracts onto $C_G(E)$,
since $\rho/\rho^+\cong{Z/2Z}$ and so
$\pi/\langle\langle\rho^+\rangle\rangle\cong{C_G(E)}$.
\end{proof}

If $M$ is a closed irreducible 3-manifold then there is
a maximal finite family $\mathcal{P}$ of pairwise non-parallel, 
2-sided copies of $RP^2$ embedded in $M$.
If $M\not=RP^2\times{S^1}$ then the orientation covers of 
the components of $M\setminus\cup\mathcal{P}$
are punctured aspherical 3-manifolds.
This decomposition gives rise to a graph-of-groups structure for $\pi$
in which each edge group has order 2 and each vertex group has one end. 
(See \cite{Ep,He,Swa}.)

\begin{cor*}
Let $M$ be an irreducible closed $3$-manifold such that $\pi_1(M)$ 
has an element of finite order $>1$ and infinite centralizer.
Then $M\cong{RP^2\times{S^1}}$.
\end{cor*}

\begin{proof}
We may use the graph-of-groups structure for $\pi$ determined by
copies of $RP^2$ in $M$.
Suppose that $M$ is not homeomorphic to $RP^2$.
Then it is not homotopy equivalent to $RP^2$ \cite{Ta,To},
and so it splits along a 2-sided embedded $RP^2$ into two pieces 
(corresponding to the decomposition of $\pi$ as a 
generalized free product of $\rho$ and $Z/2Z\oplus\mathbb{Z}$ 
with amalgamation over $Z/2Z$), 
each with boundary $RP^2$. 
But this is impossible, since $RP^2$ does not bound.
\end{proof}

The argument for the corollary does not extend to $PD_3$-complexes,
since there is as yet no result comparable to the Projective Plane Theorem, 
allowing splitting along 2-sided copies of $RP^2$.

\smallskip
\noindent{\bf Question 1.} {\sl If $P$ is an indecomposable $PD_3$-complex 
and $\pi=\pi_1(P)\cong\rho*_{Z/2Z}(Z/2Z\oplus\mathbb{Z})$
must $P\simeq{RP^2}\times{S^1}$?}

One possible approach to this question might be to apply the Turaev Criterion and its consequences \cite{Tu}. (See also \cite[Corollary 2.4.1]{Hi}.)

A more ambitious approach, putting this curiosity into a wider context,
would be to extend Turaev's Splitting Theorem,
which  establishes one of the major consequences 
of the  Sphere Theorem for 3-manifolds.
We might hope that the following holds.

\smallskip
\noindent{\bf Question 2.} {\sl Let $P$ be a $PD_3$-complex 
with fundamental group $\pi=\pi_1(P)$ and orientation character $w=w_1(P)$.
If $\pi$ splits over a subgroup $C$ such that $w|_C$ is injective then 
$P$ is either a connected sum or splits over a two-sided copy of $RP^2$.}

This is so when $C$ is trivial, by Turaev's Splitting Theorem \cite{Tu}. 
(See also \cite[Theorem 2.7]{Hi}.)

A relative version is desirable for inductive arguments.

\smallskip
\noindent{\bf Question 3.} {\sl Let $(X,\partial{X})$ be a $PD_3$-pair 
such that $\pi^+$ is torsion-free,
and let $b:\partial{X}\to{X}$ be the inclusion of the boundary.
Then $\mathrm{Cok}(\pi_2(b))\not=0$ if and only if either $X$ is a proper connected sum, $X$ splits over a two-sided copy of $RP^2$, 
or $\pi$ has two ends.}

If $P$ is a $PD_3$-complex such that $\pi_1(P)$ is infinite 
and has non-trivial torsion then $\pi_2(P)\not=0$, 
and so a positive answer to Question 3 would confirm Question 2 also.

We note finally that if $P$ is an indecomposable $PD_3$-complex and
$g\in\pi=\pi_1(P)$ has centralizer $C$ of finite order $>2$ then 
$P$ is orientable and $C$ is cyclic or has cohomological period 4.
In the 3-manifold case $\pi$ must be finite, 
but if $C$ is any finite group of cohomological period 4  
(and hence of even order $\geq6$)
and $D_{2m}$ is the dihedral group of order $2m$ then 
$C*_{Z/2Z}D_{2m}$ is the fundamental group of a $PD_3$-complex.
(See \cite[Chapters 6 and 7]{Hi}.)

\newpage
\section{responses}

J.  S. Crisp has suggested an alternative formulation for the Lemma,
that is consistently in terms of groups acting on trees, 
and avoids normal form arguments.

Let $U$ be a $G$-tree with one orbit of edges, 
corresponding to the HNN extension.
Fix a base edge $e$ with vertices $t^{-1}(v)$ and $v$, 
and with $Stab(e)=\langle{g}\rangle$.
Then $Stab(t(e))=\langle{tgt^{-1}}\rangle=\langle{j}\rangle$.
Let $A=Stab(v)$ and $C=C_G(g)$.
The subgraph $T=Fix(g)$ is a $C$-tree, 
with edges in bijection with the cosets of $Stab(e)$ in $C$.
If $T$ has more than one edge then it has another edge $e'$ 
which is adjacent to $e$.
We may write $e'=z(e)$, for some $z\in{C}\setminus\langle{g}\rangle$.
Then there are four possibilities
\begin{enumerate}

\item $z(v)=v$: so $z\in{C_A(g)}$;

\item $z(t^{-1}(v))=t^{-1}(v)$: so $tzt^{-1}\in{C_A(j)}$;

\item $z(t^{-1}(v))=v$: so $w=zt^{-1}\in{A}$ and  $wgw^{-1}=j$;
or 

\item $z(v)=t^{-1}(v)$: similar.
\end{enumerate}
In each case one of the hypotheses is contradicted.
Thus $T$ must have just one edge, 
so $C=Stab(e)=\langle{g}\rangle$.

He has also suggested a simplification of the proof of the theorem,
avoiding the lemma entirely.

Let $P$ be a $PD_3$-complex and $\pi=\pi_1(P)$,
and let $X$ be a terminal $\pi$-tree.
Suppose that $X$ has an edge $e$ such that $Stab(e)=\langle{g}\rangle$,
where $g$ has order 2 and $C_\pi(g)/\langle{g}\rangle\cong\mathbb{Z}$.
Then $Fix(g)$ is an infinite line connecting two ends of $X$
 \cite[Theorem 17]{Cr}. 
(See also \cite[Theorem 4.9]{Hi}.)
We may contract $X$ onto the $\pi$-orbit of $e$ to obtain a $\pi$-tree $T$ with an edge-transitive action, and $Fix(g)$ embeds in $T$.
If $f=h(e)$ is an edge in $Fix(g)$ then $gh=hg$ and so $h\in{C_\pi(g)}$.
Since we may assume that $e$ and $f$ are adjacent,
 $C_\pi(g)$ acts by unit translations, 
and so the HNN structure for $\pi$ has a stable letter  $u\in{C_\pi(g)}$.
The theorem now follows, as above.

G.  A.  Swarup has pointed out that the 3-manifold Corollary follows 
from the Compact Core Theorem of Scott and the fact
that if a compact non-orientable 3-manifold 
$N$ has finite fundamental group then 
$N\simeq{RP^2\times[0,1]\setminus{F}}$, 
for some finite subset $F$ \cite{Ep}.

Let $M$ be an irreducible closed 3-manifold such that $\pi=\pi_1(M)$ 
has a subgroup $H\cong{Z/2Z}\oplus\mathbb{Z}$,
and let $M_H$ the the associated covering space.
Then $M_H$ has a compact submanifold $L$ with $\pi_1(L)\cong{H}$,
by the Compact Core Theorem.
The element of order 2 in $H$ is orientation-reversing,
by Crisp's Centralizer Theorem.
We may assume that $\partial{L}$ is incompressible.
It contains no 2-spheres, since $M$ is irreducible,  and can have no aspherical components, since $\pi$ is too small.
Hence $\partial{L}\cong{m}RP^2$, for some $m\geq0$.
Let $L^+$ be the orientable double cover of $L$.
Then $\pi_1(L^+)\cong\mathbb{Z}$ and $\partial{L^+}\cong{m}S^2$.
Let $\widehat{L}=L^+\cup{m}D^3$ be the closed 3-manifold 
obtained by capping off the components of $\partial{L^+}$.
The  covering involution of $L^+$ over $L$ extends 
to an orientation inversing involution $\tau$ of $\widehat{L}$ 
with $m$ isolated fixed points,  all of local index 1.
Hence $m=Lef(\tau)=2(1-\mathrm{tr}\, H_1(\tau))$,
by the Lefschetz-Hopf Fixed Point Theorem.
But $\mathrm{tr}\, H_1(\tau)=1$, and so $m=0$.
Thus $L$ is closed,  so $H$ has finite index in $\pi$ and $\pi$ has two ends.
Hence $\pi\cong\mathbb{Z}$, 
$Z/2Z\oplus\mathbb{Z}$ or $D_\infty=Z/2Z*Z/2Z$.
Clearly $\pi\cong{Z/2Z}\oplus\mathbb{Z}$ is the only possibility,
and so $M\cong{RP^2}\times{S^1}$, as in the Corollary above.


\end{document}